\documentclass{amsart}
\usepackage{amssymb
}

\newcommand{\ncmnd}{\newcommand}
\ncmnd{\nthm}{\newtheorem}
\ncmnd{\R}{\mathbf{R}}

\theoremstyle{plain}

\nthm{thm}{Theorem}
\nthm{cor}[thm]{Corollary}
\nthm{prop}[thm]{Proposition}
\nthm{lem}[thm]{Lemma}

\theoremstyle{definition}
\nthm{note}[thm]{Note}
\nthm{rem}[thm]{Remark}
\nthm{defn}[thm]{Definition}
\nthm{examp}[thm]{Example}

\begin{document}

\title{Extremal problems for surfaces of prescribed topological type (1)}
\author{V. N. Lagunov}
\author{A. I. Fet}
\thanks{Russian version in Siberian Math. J., Vol 4, 1963, pp. 145-176. Translation and remarks enclosed in brackets [ ] are by Richard L. Bishop, University of Illinoi at Urbana-Champaign. Many parts have been abbreviated and in some case alternative proofs(?) were devised emphasizing intuition. Horizontal lines indicate original pages.}
\maketitle
\section{Introduction}\label{sec:introduction}
The study of extremal properties of surfaces with bounded smooth curvature shows that geometrical properties in the large are solidly connected to their topological structure. We take up some of these questions here.

Notation: $F_R$ -- all compact $C^2$ $n$-dimensional surfaces
contained in $E^{n+1}, n \ge 2$, with principal radii of curvature
all $\ge R$. \cite{L1,L2,L3,L4}; let 
$$\kappa(F_R) = \inf \{\rho : \textrm{ there is a ball of radius }
\rho \textrm{ interior to a surface in } F_R\}.$$
It was shown that $\kappa (F_R) = \kappa_0R$, where $\kappa_0 =
2/\sqrt 3 -1 \cong 0.155$.

The sharpness of this bound was shown by constructing examples
of surfaces $F_{(\epsilon)} \in F_R$, containing spheres of radii
$\kappa_0R+\epsilon$, for $\epsilon > 0$ arbitrarily small. The
surfaces $F_{(\epsilon)}$ have nonzero Betti numbers and bound
a body of complicated topological structure; precisely,
$F_{(\epsilon)}$ is homeomorphic to a boundary $S^n_k$ of a
ball with $k$ handles $h^{n+1}_k$, but bounds a solid not
homeomorphic to $h^{n+1}_k$ (precisely, see below, p. 188).
There remains the question of whether the above bound can
be improved if instead of $F_R$ some subset of $F_R$ is
considered, consisting of surfaces of sufficiently simple 
topological structure of sufficiently simple imbedding in
$E^{n+1}$. Some results in this direction were already presented
by us in the Second All-Union topological conference in Tbilisi in
1959 \cite{LF}.

We introduce notation: For $F\in F_R$ let $\kappa(F)$ be the radius
of a maximal ball interior to $F$. $M \subset F_R$, $\kappa(M) =
\inf_{F\in M} \kappa(F)$.

Let $\mathcal{S}$ be the subset of surfaces in $F_R$ homeomorphic
to $S^n$; $\mathcal{H}_k$ those homeomorphic to a sphere with $k$
handles, $S^n_k$;

\medskip
\hrule\smallskip
\noindent p.~146.
\medskip

$\mathcal{H}^0_k$ those which bound a solid ball with $k$ handles,
$h^{n+1}_k$; $Tr^0$ those which bound a solid toroidal ring, cf.
\S \ref{sec:applications}, part 1.
$$\kappa_1 = \sqrt{3/2} -1 \cong 0.2246.$$
Then
\begin{thm} \label{Theorem 1} If the first Betti number of $F$ mod 2 is
zero, then $\kappa(F) \ge \kappa_1R$.
\end{thm}
In case $H_1(F,Z_2) \ne 0$ we turn to the universal covering
solid $T$ of the boundary $F$; in connection with this a condition
is included on the homotopy type.
\begin{thm}\label{Theorem 2} Let $F \in F_R$. If the homomorphism
$h: \pi_1(F) \to \pi_1(T)$ induced by the inclusion of $F$ in
$T$ is an isomorphism, and $\pi_2(T) = 0$, then $\kappa(F) \ge
\kappa_1R$.
\end{thm}
For $n=2$ the condition of Theorem \ref{Theorem 2} can be weakened.
\begin{thm}\label{Theorem 3} Let $n=2$, $F \in F_R$, and $h: \pi_1(F)
\to \pi_1(T)$ be onto. Then $\kappa(F) \ge \kappa_1R$.
\end{thm}
From Theorems \ref{Theorem 1}, \ref{Theorem 2}, in combination with corresponding examples:
\begin{thm}\label{Theorem 4} $\kappa(\mathcal{S}) \ge \kappa_1R$; $\kappa(\mathcal{H}^0_k)
\ge \kappa_1R$; $\kappa(Tr^0) \ge \kappa_1R$.

For $n=2$ these inequalities reduce to equality.
\end{thm}
Examples proving the second assertion of Theorem \ref{Theorem 4} will be
constructed in the second part of this work; Theorems \ref{Theorem 1} -- \ref{Theorem 3}
and the first assertion of Theorem \ref{Theorem 4} are proved in this first part.

Sharp bounds in Theorem \ref{Theorem 4} for $n>2$ are unknown. We note that
$\kappa(\mathcal{H}_k) = \kappa_0R, \ k=1,2,\cdots$, cf. \cite{L3},
Introduction. This shows that a surface homeomorphic to a
sphere with $k$ handles and bounding a solid ``sufficiently
correctly'' in a topological sense contains a ball of radius
$\kappa_1R$; but surfaces can be constructed for which the
topological type of the body bounded is ``incorrect'', which
contain only balls of radius differing from $\kappa_0R$ by an
arbitrarily small amount. The ``critical numbers'' $\kappa_0$
and $\kappa_1$ have a simple geometrical meaning: $\kappa_0R$
is the radius of the greatest circle in the plane included between
three tangent circular arcs of radius $R$; $\kappa_1R$ is the
radius of the greatest ball in $E^3$ included between four
tangent spheres of radius $R$.

 In \cite{LF} the equation $\kappa(\mathcal{S}) = \kappa_1R$ was published
for the case $n=2$ of Theorem \ref{Theorem 3}; there also was indicated the
possibility of generalizing these results for any $n$.

In this work we depend on geometric methods developed in
\cite{L3} and in a series of results of the work \cite{L3} assumed
to be known.

\S \ref{sec:geom} is carried out with a purely geometric character. In it is
established Lemma \ref{Lemma 2}, from which the proofs of our theorems
upon establishing that the multiplicity of the central set $Z$
is greater than three (cf. \cite[p. 225 (3:5)]{L3}). As proved in
\cite{L3}, the multiplicity of $Z$ must be greater than 2;
consequently, it remains to obtain conditions on $F$ and $T$
excluding multiplicity 3 and that it then follows that $\kappa(F)
\ge \kappa_1R$. In \S \ref{sec:local} the local structure of $Z$ is studied under
the assumption that the multiplicity is 3. In \S \ref{sec:triang} it is proved
that $Z$ (in the case of multiplicity 3) has a topological
structure defined there and called a 3-complex.

\medskip
\hrule\smallskip
\noindent p.~147.
\medskip

In \S\S \ref{sec:cover} -- \ref{sec:graph} the topological properties of a 3-complex are
studied, abstracting from the fact that a 3-complex is a
central set of $T$; in these paragraphs only the basic topological
properties of the configuration $\{ F, T, Z \}$ are used which
are recounted at the beginning of \S \ref{sec:triang}. In the results it is
clarified what the topological conditions are needed on $F,T$
in order that a 3-complex $Z$ will fail to exist (Lemma \ref{Lemma 16}).
The proof of our theorems are completed in \S \ref{sec:proofs} by combining
the results of \S \ref{sec:triang} (cf. above) with Lemma \ref{Lemma 16}.

\section{Geometrical lemmas}\label{sec:geom}
\noindent 1. Let $g_1, \cdots, g_k$ be unit vectors in $E^{n+1}$,
$\beta = $ the minimum angle between pairs of them, and
$\alpha^{n+1}(k)$ the supremum of such $\beta$ (cf. \cite[p. 226]{L3}). We need
\begin{lem}\label{Lemma 1} $\alpha^{n+1}(4) = 2 \csc^{-1}\sqrt{3/2}$.
(Equality with $\beta$ occurs for the vectors which go from the
center of a regular tetrahedron to its vertices.)
\end{lem}
[A proof due to Reshetnyak is given.]

\noindent 2. Let $F^n \in F_R$ bound a body $T^{n+1}$; we designate by $Z$
the central set [cut locus] of $F^n$ \cite[p. 224]{LF}. In the
following it is assumed everywhere that $F^n$ is a flattened
surface, and consequently, $Z$ has multiplicity $>2$ (cf. \cite[pp. 206, 225]{L3}). Such assumptions do not limit the generality of
considerations, since for nonflattened $F^n$ the results of this
work are evident, but for flattened ones the multiplicity of
$Z$ is $>2$. (\cite[pp. 231-232]{L3}).
\begin{lem}\label{Lemma 2} If the multiplicity of $Z$ is $>3$, then
$F^n$ contains a sphere of radius $\kappa_1R$.
\end{lem}
[In the proof Lemma \ref{Lemma 1} is applied to 4 unit vectors going from a
point of multiplicity $\ge 4$ along lines normal to $F^n$.]

\medskip
\hrule\smallskip
\noindent p.~148.
\medskip

\section{Local structure of the central set}\label{sec:local}
[Standard properties of the cut locus of multiplicity 3 are
developed. Cf. Ozols paper on that subject. They are the properties
abstracted as a normally imbedded 3-complex in \S \ref{sec:triang}.]
\section{Triangulations and 3-complexes}\label{sec:triang}

\medskip
\hrule\smallskip
\noindent p.~152.
\medskip

\noindent 1. Continuing we will need some properties of $Z$ shared with
$\tilde Z$, the covering of $Z$ in the universal covering
$\tilde T$ of $T$. To avoid repetition and provide convenient
reference we formulate a {\em 3-complex $Z^n$ in} $T^{n+1}$ as
satisfying:
\begin{description}
\item[1)] $Z^n$ is an $n$-dimensional locally finite polyhedron,
triangulated by $\tau$.
\item[2)] $Z^n$ contains a subcomplex $Z^{n-1}$, decomposed
into a finite or countable nonoverlapping union of
$(n-1)$-manifolds $Z^{n-1}_i$.
\item[3)] $Z^n\setminus Z^{n-1}$ is a finite or countable
union of $n$-dimensional manifolds.
\item[4)] Each $(n-1)$-simplex of $Z^{n-1}$ is a face of
exactly 3 $n$-simplices of $Z^n$.
\end{description}

We say that the 3-complex $Z^n$ is {\em normally imbedded in
$T^{n+1}$} if 5)--12) as follows hold.
\begin{description}
\item[5)] $T^{n+1}$ is an $(n+1)$-manifold with boundary $F^n$.
($T^{n+1}$ is not generally compact, $F^n$ not generally
connected.)
\item[6)] $Z^n$ is a closed subset of $T^{n+1}\setminus F^n$.
\item[7)] The triangulation $\tau$ is extended to one of
$T^{n+1}$, $\tau^0$, for which $F^n$ is a subcomplex.
\item[8)] $Z^n$ is a deformation retract of $T^{n+1}$ by
$\varphi_t : T^{n+1} \to T^{n+1}$ with $\varphi_1 : T^{n+1} \to
Z^n$ and $\varphi = \varphi_1|F^n$ is simplicial.
\item[9)] $F^n$ is a deformation retract of $T^{n+1} \setminus
Z^n$ by $\psi_t: T^{n+1} \setminus Z^n \to T^{n+1} \setminus Z^n$,
deforming the identity $\psi_0$ to $\psi_1 = \psi : T^{n+1}
\setminus Z^n \to F^n$.
\item[10)] $\varphi$ is a 2-fold covering on $\varphi^{-1}(Z^n
\setminus Z^{n-1})$.
\item[11)] $\varphi$ is a 3-fold covering on $\varphi^{-1}(Z^{n-1})$.
Moreover, each $Q \in Z^{n-1}$ has a neighborhood $W$ such that
[$Z^n \cap W$ is a triad bundle over $Z^{n-1}$].
\item[12)] The $Z^n$-star (closed) of each vertex $Q$ in $Z^n$
belongs to a neighborhood $W(Q)$ evenly covered by $\varphi$
(cf. 10)) if $Q \in Z^n \setminus Z^{n-1}$, or decomposed as in
11) if $Q \in Z^{n-1}$. In each component of $Z^n \setminus
Z^{n-1}$ there is at least one vertex $Q$ for which the closed
star doesn't meet $Z^{n-1}$.
\end{description}

\medskip
\hrule\smallskip
\noindent p.~153.
\medskip

\noindent {\em Remark.} The properties are not independent; for
example 4) follows from 11).

\noindent 2.

\begin{lem}\label{Lemma 4} If the central set $Z^n$ of a body
$T^{n+1}$ of $(n+1)$-dimensional Euclidean space, bounded by a
surface $F^n \in F^n_R$, has multiplicity 3, then $Z^n$ is a
3-complex normally imbedded in $T^{n+1}$.
\end{lem}

For the triangulation use the methods of Whitney \cite[pp. 175-191]{W}.

\medskip
\hrule\smallskip
\noindent p.~154.
\medskip

The rest of the proof has been set up by the preceding material.

In the continuation the triangulation of $T^{n+1}$ is assumed
to extend triangulations of $Z^n, F^n$ so that $\varphi : F^n
\to Z^n$ is simplicial.

\noindent 3. We construct for the polyhedron $T^{n+1}$ of part 2 the
universal covering $\kappa : \tilde T^{n+1} \to T^{n+1}$.
$\tilde F^n = \kappa^{-1}(F^n)$ is the boundary of the manifold
$\tilde T^{n+1}$, $\tilde Z^n = \kappa^{-1}(Z^n)$ is the universal
covering of $Z^n$, and the deformation retracts $\varphi_t,
\psi_t$ can be lifted.

\medskip
\hrule\smallskip
\noindent p.~155.
\medskip

The triangulation can be lifted too, so
\begin{lem}\label{Lemma 5} $\tilde Z^n$ is a 3-complex normally
imbedded in $\tilde T^{n+1}$.
\end{lem}

\section{Coverings in 3-complexes}\label{sec:cover}

\noindent 1. Consider a normally imbedded 3-complex $Z^n \subset T^{n+1}$.
Denote connected components by subscripts: $Z^{n-1}_j, Z^n_i,
F^{n-1}_j, F^n_i$. The closures of the $n$-dimensional ones are
subcomplexes. $F^{n-1} = \varphi^{-1}(Z^{n-1})$, all have
triangulated closures.
\begin{lem}\label{Lemma 6} 
\begin{description}
\item[1)] If $P\in F^n \setminus F^{n-1}$ then
$\varphi(P)$ is a double point of $Z^n$.
\item[2)] If $P \in F^{n-1}$, then $\varphi(P)$ is a triple point.
\item[3)] For each $F^{n-1}_j$, $\varphi(F^{n-1}_j)$ coincides
with some $Z^{n-1}_{k_j}$ and
$\varphi : F^{n-1}_j \to Z^{n-1}_{k_j}$ is a covering.
\item[4)] For each $Z^{n-1}_k$ there is at least on $F^{n-1}_j$
such that $k = k_j$.
\item[5)] If $F^{n-1}_j$ is oriented for all $j$ such that
$k = k_j$, then so is $Z^{n-1}_k$.
\end{description}
\end{lem}

[Of these only 5) seems to need explaining. Since $\varphi :
\varphi^{-1}(Z^{n-1}_k) \to Z^{n-1}_k$ is a 3-fold covering,
the restrictions to components of $\varphi^{-1}(Z^{n-1}_k)$
must be coverings whose multiplicities add up to 3. One of the
multiplicities must be odd (1 or 3), so that $Z^{n-1}_k$ is
oriented.]

\noindent 2. [In this part some combinatorics of simplices are developed.
It is a clumsy but precise way of getting the essential properties
of tubular neighborhoods of the $Z^{n-1}_j$. I believe a better
alternative is to use cells rather than simplices, and adapt the
cells to the local product structure of the triad bundle.]

\medskip
\hrule\smallskip
\noindent p.~157.
\medskip

\noindent 3. [ A similar development is given for tubular neighborhoods of
$Z^n_i$ in $T^{n+1}$.]

\medskip
\hrule\smallskip
\noindent p.~158.
\medskip

\noindent 4. [More combinatorics.]

\medskip
\hrule\smallskip
\noindent p.~159.
\medskip

\noindent 5. [The idea of the holonomy of the triad bundle is pursued
using the combinatorics of the previous parts. A component
$Z^{n-1}_j$ is said to be a {\em manifold of the first class} if
the holonomy is trivial. It is said to be of {\em of the second
class} if the holonomy consists of a group of order 2, so that
two arms of the triad can be transposed and neither is connected
to the third arm. If the holonomy group is transitive on the three
arms, it is said to be {\em of the third class}.

\medskip
\hrule\smallskip
\noindent p.~160. 
\medskip

The finer
classification of the third class into those with holonomy the
alternating subgroup of the three arms and those with holonomy
all permutations of the three arms is not discussed. Probably the
latter is ruled out later by orientability considerations, along
with those of the second class.]

\section{Basic topological lemmas}\label{sec:basic}

\noindent 1. We consider homology groups $H_q(M,G)$ using $G = J$, the
integers, and $G=J_2$, the integers mod 2. For infinite but
locally finite complexes there are further homology theories:
$H^{fin}_q(M,G)$, the homology of finite chains, and
$H^{inf}_q(M,G)$, the homology of infinite chains. The basic
reference is \cite[\S 9]{E}. The symbol $\sim$ is used to
denote ``homologous''.
\begin{lem}\label{Lemma 11} If $H^{inf}_{n-1}(T^{n+1},J) = 0$ and all
$Z^{n-1}_j$ are orientable, then manifolds of the third class
don't exist.
\end{lem}
\begin{proof} Since $Z^n$ is a deformation retract of $T^{n+1}$,
we also have that $H^{inf}_{n-1}(Z^n,J)=0$. Since $Z^{n-1}_j$ is
orientable, it is a cycle for a chosen orientation. But then it
must be a boundary in $Z^n$, $Z^{n-1}_j = \partial c^n$ for some
$n$-chain of $Z^n$.

If $Z^{n-1}_j$ is of the third class, then any [$n$-cell] adjacent
to $Z^{n-1}$ [has a coefficient in $c^n$ which must propogate to
adjacent cells in a tubular neighborhood of $Z^{n-1}_j$,
continuing to all of the cells adjacent to $Z^{n-1}_j$. If the
$n$-manifold formed by these cells is orientable, then the
boundary of $c^n$ must have every $(n-1)$-cell of $Z^{n-1}_j$
with coefficient which is a multiple of 3. If the $n$-manifold
is nonorientable, then the holonomy contains a transposition
and the boundary of $c^n$ could not have all $(n-1)$-cells of
the interior of the $n$-manifold cancel, so the boundary could
not be the fundamental class of $Z^{n-1}_j$.]
\end{proof}

\medskip
\hrule\smallskip
\noindent p.~161.
\medskip

\noindent 2.
\begin{lem}\label{Lemma 12} If $T^{n+1}$ and $Z^{n-1}$ are orientable,
then there are no manifolds of the second class.
\end{lem}
[If the holonomy has a transposition, then the normal bundle of
$Z^{n-1}$ is nonorientable, so just one of $T^{n+1}$ and $Z^{n-1}$
is orientable along the loop giving that transposition.]

\noindent 3. We turn to the study of manifolds of the first class. If
$Z^{n-1}_j$ is of the first class, then a [tubular neighborhood
of $Z^{n-1}$ with $Z^{n-1}$ removed] has three connected
components, the {\em 3 sheets adjacent to $Z^{n-1}_j$}.
[Another Lemma, omitted, formulates this in terms of the combinatorics of
simplices.]

\medskip
\hrule\smallskip
\noindent p.~163.
\medskip

\noindent 4. Designate the 3 sheets in a tubular neighborhood $U$ of
$Z^{n-1}_j$ by $M_{j\alpha}, \alpha = 1,2,3$. The closures of
these $U$ are assumed to be disjoint.
\begin{lem}\label{Lemma 14} Let $Z^{n-1}_j$ be of the first class.
Then
\begin{description}
\item[1)] Two distinct $M_{j\alpha}$'s have no interior points
of $Z^n$ in common. The only component of $Z^{n-1}$ having
points in the closure of $M_{j\alpha}$ is $Z^{n-1}_j$.
\item[2)] Each manifold $Z^n_i$ has at least one point not
belonging to any ${\bar M}_{j\alpha}$.
\item[3)] Each ${\bar M }_{j\alpha}$ belongs to a unique
${\bar Z}^n_i$ and does not have interior points in common
with any other ${\bar Z}^n_k$.
\item[4)] For the $n$-chain mod 2 carried by $M_{j\alpha}$,
also designated by $M_{j\alpha}$, the boundary is
\begin{equation}\label{eq:lem14}\partial M_{j\alpha} = Z^{n-1}_j + z^{n-1}_{j\alpha}
\end{equation}
where $Z^{n-1}_j$ is the fundamental $(n-1)$-cycle mod 2 of
$Z^{n-1}_j$ and $z^{n-1}_{j\alpha}$ is a cycle mod 2, not 0,
and having no common points with $Z^{n-1}$.
\end{description}
\end{lem}
[Again this is given and proved in terms of the simplicial
triangulation.

In terms of the structure of bundles over the $Z^{n-1}$ with
triad bundle, the assumption that $Z^{n-1}_j$ is of the first
class tells us that the bundle is trivial, so that
${\bar M}^{n-1}_{j\alpha} = Z^{n-1}_j \times [0,1]$, and in
these terms 4) is geometrically transparent:
$$\partial M^{n-1}_{j\alpha} = Z^{n-1}_j \times \{0\} +
Z^{n-1}_{j\alpha} \times \{1\}.]$$

\medskip
\hrule\smallskip
\noindent p.~165.
\medskip

\noindent 5.
\begin{lem}\label{Lemma 15} Let $F^n$ be orientable and
$H^{fin}_1(F^n,J_2) = 0$; then all $Z^{n-1}_j$ are orientable.
\end{lem}
[The proof given invokes Poincar\'e duality (\cite[\S 33]{E})
in the form $H^{fin}_1(F^n,J_2) \approx H^{inf}_{n-1}(F^n,J_2)$,
and goes on to argue that $(n-1)$-submanifolds $F^{n-1}_j$ of
$F^n$ are orientable. Then since $\varphi ; F^{n-1}_j
\to Z^{n-1}_j$ is a 3-fold cover, $Z^{n-1}_j$ must be orientable
too.

We can avoid the use of Poincar\'e duality by a more direct
argument to show that $F^{n-1}_j$. Suppose we have a loop
$\gamma$ in $F^{n-1}_j$. What $H^{fin}_1(F^n,J_2) = 0$ means
is that $\gamma = \partial c^2$, where $c^2$ is a finite 2-chain
mod 2. Hence $c^2$ is carried by a compact immersed 2-manifold
$S$ with boundary. We can put $S$ in general position relative
to $F^{n-1}_j$, which means that the intersection is a graph
including $\gamma$ in such a way that vertices on $\gamma$
are all triple points and there are no other branch points. Using
this graph we decompose $\gamma$ into a sum of simple cycles
along which $S$ provides a normal field to $F^{n-1}_j$. Since
$F^n$ is orientable, these simple cycles preserve orientation
on $F^{n-1}_j$, and hence so does $\gamma$.]

\section{Representing graph}\label{sec:graph}
\noindent 1. The complex $Z^n$ is built from subcomplexes $Z^n_i$,
attached to one another by subcomplexes $Z^{n-1}_j$; for a more
detailed study of this situation we construct the {\em representing
graph} $\Gamma$ of the 3-complex $Z^n$. This has two kinds of vertices:

$e_i$ -- principal vertices, one for each $Z^n_i$;

$\epsilon_j$ -- auxiliary vertices, one for each $Z^{n-1}_j$;

\noindent and edges $k_{j\alpha}$ corresponding to the sheets
$M_{j\alpha}$ and joining a principal vertex to an auxiliary
vertex if and only if the sheet of the auxiliary vertex
$\epsilon_j$ is contained in the $Z^n_i$ corresponding to the
principal vertex $e_i$.

There are just 3 edges ending in each auxiliary vertex; even
if some of the 3 sheets coincide we still take 3 edges [but see
the next paragraph].

Somewhat retreating from the customary definition of a graph,
we call the set of all vertices and edges of $\Gamma$ the
{\em representing graph} of the 3-complex $Z^n$. We note that
manifolds $Z^{n-1}_j$ of the second and third class do not
play a r\^ole in the preceding definition, which will be used
only under conditions guaranteeing the nonexistence of such
manifolds.

\noindent 2. We say that a subgraph $\Gamma' \subset \Gamma$ is a
{\em proper tree} if $\Gamma'$ has no cycles and each auxiliary
vertex of $\Gamma'$ is incident with exactly two edges.
[This definition seems incomplete: I think they intend to
include connectedness and/or maximality with respect to
the specified properties.]

\medskip
\hrule\smallskip
\noindent p.~166.
\medskip

\begin{lem}\label{Lemma 16} $\Gamma$ has either a cycle or a
proper tree.
\end{lem}
\begin{proof} Build $\Gamma'$ recursively as an increasing
union of connected subgraphs. Start with $\Gamma_1$
consisting of a single principal vertex, all of the edges from
it, and the auxiliary vertices at the other end of those edges.

Stop whenever a cycle is obtained. Otherwise get
$\Gamma_{\mu +1}$ from $\Gamma_\mu$ by choosing a second
edge for each auxiliary vertex of $\Gamma_\mu$ which has
no second edge, add in the other ends of those new second
edges, and add in all the edges (and their ends) incident to
the new principal vertices.

\medskip
\hrule\smallskip
\noindent p.~167.
\medskip

In this process, if we are forced to take into $\Gamma_{\mu+1}$
the third edge of some auxiliary vertex already in $\Gamma_\mu$,
then within $\Gamma_{\mu+1}$ there are two distinct paths
from the starting vertex to the auxiliary vertex in question:
one in $\gamma_\mu$ and one in $\Gamma_{\mu+1}$ using
the third edge. Hence $\Gamma_{\mu+1}$ must contain a cycle.

Taking $\Gamma' = \bigcup_\mu \Gamma_\mu$, either
$\Gamma'$ has a cycle or it is a proper tree such that
\begin{description}
\item[1)] $\Gamma'$ is connected and
\item[2)] whenever a principal vertex belongs to $\Gamma'$,
then so do all the edges incident to it.
\end{description}
\end{proof}

\medskip
\hrule\smallskip
\noindent p.~168.
\medskip

\noindent 3. Corresponding to each cycle $\Gamma_0 \subset \Gamma$
we construct a 1-cycle with compact support in the polyhedron
$Z^n$. Let $\Gamma_0$ consist of edges $k_{j_s\alpha_s},
s=1,\ldots, t$; $t$ is even (equal to twice the number of
principal vertices incident to the edges of $\Gamma_0$).
Let the numbering of the edges of $\Gamma_0$ be carried
out so that the ends of $k_{j_s\alpha_s}$ are the principal
vertex $e_{i_s}$ and the auxiliary vertex $\epsilon_{j_s}$,
$\epsilon_{j_s} = \epsilon_{j_{s+1}}$ for $1 \le s < t, s$ odd,
$e_{i_s} = e_{i_{s+1}}$ for $1<s<t, s$ even, $e_{i_t} = e_{i_1}$.
For convenience in writing out we will understand by
$k_{j_{t+1},\alpha_{t+1}}, \epsilon_{j_{t+1}}$, respectively,
$k_{j_1,\alpha_1}, \epsilon_{j_1}$.
\begin{lem}\label{Lemma 17} For each vertex of $\Gamma_0$ choose
a point in the interior of the corresponding submanifold $Z^n_i$
or $Z^{n-1}_j$. For each edge of $\Gamma_0$ with ends $e_i$
and $\epsilon_j$ choose a path in $Z^n_i$ from $e_i$ to
$\epsilon_j$ which contains points of only the corresponding
sheet of $Z^{n-1}_j$ besides points in $Z^n_i$ -- no other sheets
of $Z^{n-1}_j$ nor any other $Z^{n-1}_k, j \ne k$. Then these
paths will form a loop in $Z^n$, and if the loop in $\Gamma_0$
is simple, the loop in $Z^n$ can be chosen to be simple as well.
Moreover, at points of $Z^{n-1}$

\medskip
\hrule\smallskip
\noindent p.~169.
\medskip

\noindent on this loop in $Z^n$ the loop passes from one sheet of
$Z^{n-1}_j$ to another, and (if simple) can never hit $Z^{n-1}_j$
again because $\epsilon_j$ can only occur once in the loop of
$\Gamma_0$.
\end{lem}

This means that each $Z^{n-1}_j$ crossed by one of these loops
does not separate $Z^{n-1}_j$. [What if $n=2$ and $Z^{n-1}_j$ is
not closed?]

\noindent 4. For a proper tree $\Gamma' \subset \Gamma$ we construct
an $n$-dimensional submanifold $M'\subset Z^n$. $M'$ consists
of the union of ${\bar Z}^n_i$ for which $e_i$ is a vertex of
$\Gamma'$. By the requirement that $\Gamma'$ has all of the
edges attached to such an $e_i$, all of the boundary of $Z^n_i$
is contained in $M'$. By the requirement that each $\epsilon_j$
in $\Gamma'$ is incident to exactly two edges in $\Gamma'$,
$M'$ is a manifold in a neighborhood of each point of $Z^{n-1}_j$,
since exactly two of the 3 sheets along $Z^{n-1}_j$ are contained
in $M'$.
\begin{lem}\label{Lemma 18} Clearly $M'$ forms an $n$-cycle mod 2 in
$T^{n+1}$. If $F^n$ is connected, then $M'$ does not separate
$T^{n+1}$. Indeed, starting at a point $A$ of $M'$ in $Z^n_i$ we
can run out on either side to the points $\varphi^{-1}(A) =
\{ A', A''\} \subset F^n$. Then $A', A''$ can be connected by a
path in $F^n$, closing a loop which crosses $M'$ simply.
\end{lem}

\section{Proofs of theorems 1,2,3}\label{sec:proofs}

\noindent 1. In this paragraph theorems 1,2,3 are proved, giving sufficient
topological conditions for the validity of the bound $\kappa(F) \ge
\kappa_1R$ in the class $F_R$. First we prove two lemmas.
\begin{lem}\label{Lemma 19} For a 3-complex $Z^n$ normally imbedded
in $T^{n+1}$ the following conditions cannot hold simultaneously:
\begin{description}
\item[1)] the boundary $F^n$ of $T^{n+1}$ is connected;
\item[2)] $F^n$ is orientable;
\item[3)] $T^{n+1}$ is orientable;
\item[4)] $H^{fin}_1(F^n,J_2) = 0$;
\item[5)] $H^{fin}_1(T^{n+1},J_2) = 0$; [i.e., $H^{fin}_1(Z^n,J_2)
= 0$]
\item[6)] $H^{inf}_{n-1}(T^{n+1}, J) = 0$. [i.e.,
$H^{inf}_{n-1}(Z^n,J) =0$]
\end{description}
\end{lem}
\begin{proof} We suppose 1) -- 6) hold.

a) From conditions 2) and 4) and Lemma 15 it follows that all
$Z^{n-1}_j$ are orientable.

Due to condition 6) and Lemma \ref{Lemma 11} there do not exist manifolds
$Z^{n-1}_j$ of the third class. From condition 3) and Lemma \ref{Lemma 12}
it follows that also there do not exist manifolds $Z^{n-1}_j$ of
the second class.

b) We consider the first possibility specified in Lemma 16: let
the representing graph $\Gamma$ of the 3-complex $Z^n$ contain
a cycle $\Gamma_0$. According to Lemma 17, $\Gamma_0$
corresponds to a 1-cycle $\zeta$ of the complex $Z^n$. From
condition 5) it follows that there is a finite 2-chain $c^2$ in
$Z^n$ such that $\partial c^2 = \zeta$ mod 2. (21) [Some of the numbering of equations in the original is retained.]

\medskip
\hrule\smallskip
\noindent p.~171.
\medskip

[This is my proof. We may assume $\zeta$ is a simple loop. We can realize $c^2$ as a union of immersed
compact surfaces, one of which has boundary $\zeta$, the others
without boundary. By taking them in general position we can
assume that the intersection with any $Z^{n-1}_j$ is a union of
regular curves. Since $Z^{n-1}_j$ is a closed $(n-1)$ manifold
(not necessarily compact), the intersections with these surfaces
are circles except for the one with boundary $\zeta$. Because
$\zeta$ crosses $Z^{n-1}_j$ just once, there is only one endpoint
for the intersection of that part of $c^2$ with $Z^{n-1}_j$ which
is impossible.]

[The proof given.] We take an arbitrary manifold $Z^n_i$
intersecting $\zeta$; designate by $c^2_j$ the 2-chain mod 2
consisting of all simplices of ${\bar Z}^n_i$ belonging to $c^2$.
[From part 3, \S \ref{sec:graph}, we had $\zeta_i$, the part of $\zeta$ in
${\bar Z}^n_i$.] Then from (21) we see that 
\begin{equation}\partial c^2_i = \zeta_i + c_i \label{eq:22}
\end{equation}
where $c_i \subset Z^{n-1}$. In defining ``sheets''
$M^{n-1}_{j\alpha}$ we had a cycle ``parallel'' to $Z^{n-1}_j$
forming the boundary of that sheet
$$\partial M^{n-1}_{j\alpha} = Z^{n-1}_j + z^{n-1}_{j\alpha}.$$
Let $Z^{n-1}_j$ be one of the components of $Z^{n-1}$ containing
an end of $\zeta_i$ so that the sheet $M^{n-1}_{j\alpha} \subset
{\bar Z}^n_i$. Then $\zeta$ has intersection number 1 with
$z^{n-1}_{j\alpha}$, just as it does with $Z^{n-1}_j$. We use
notation $\times$ for intersection numbers: $z^{n-1}_{j\alpha}
\times \zeta = 1$. Since $z^{n-1}_{j\alpha} \subset Zn_i \subset
Z^n \setminus Z^{n-1}$, from (\ref{eq:22}) we obtain
\begin{equation}\partial c^2_i \times z^{n-1}_{j\alpha} = 1 \text{ in } Z^n_i.
\label{eq:23}
\end{equation}

Let $C$ be the union of all the closed stars of the complex
${\bar Z}^n_i$ intersecting the support of $c^2_i$. Since $c^2_i$
is finite and $Z^n$ is a locally finite polyhedron, so also $C$
consists of a finite number of simplices.

It is clear that
\begin{equation}\partial C \cap c^2_i \cap Z^n_i = 0. \label{eq:24}
\end{equation}
[$C$ is a tubular neighborhood of $c^2_i$, so its boundary only
intersects $c^2_i$ at the ends of the tube, which lie in $Z^{n-1}$,
excluded from the open manifold $Z^n_i$.]

Let $c^{n-1}$ consist of all the simplices of $Z^{n-1}_{j\alpha}$
belonging to $C$; since $z^{n-1}_{j\alpha}$ is a cycle (generally
speaking, infinite), $\partial c^{n-1}$ is contained in the support
of $\partial C$, and from (\ref{eq:24}) we arrive at
$$\partial c^2_i \times z^{n-1}_{j\alpha} = \partial c^2_i \times
c^{n-1} = c^2_i \times \partial c^{n-1} = 0,$$
which contradicts (\ref{eq:23}).

c) We consider the second possibility specified in Lemma \ref{Lemma 16}: let
$\Gamma$ contain a proper tree $\Gamma'$. The result of part
a) allows the use of Lemma \ref{Lemma 18}. According to Lemma \ref{Lemma 18},
$\Gamma'$ corresponds to an $n$-dimensional (generally
speaking infinite) cycle $\zeta^n$ mod 2 of the polyhedron
$Z^n$. [$\zeta^n$ is an $n$-manifold. From a point on it, $A_0$,
we can move on paths on either side (locally) in $T^{n+1}$ out
to points $B_1,B_2 \in F^n$. connecting $B_1,B_2$ by an arc in
$F^n$ we get a loop $\delta$ in $T^{n+1}$ having a simple
intersection with $\zeta^n$.] $\delta$ can be represented
simplicially and we have $\delta \times \zeta^n = 1$ (27).

\medskip
\hrule\smallskip
\noindent p.~172.
\medskip

By condition 5), $\delta = \partial b^2$ mod 2 for some finite
mod 2 2-chain $b^2$. By the same argument as in b) we reach a
contradiction. [$b^2$ is essentially a compact immersed surface
with boundary $\delta$. It can be taken in general position
relative to $Z^n$, so the intersection is a regular curve. But that
curve only has one end by (27).]
\end{proof}

\noindent 2.
\begin{lem}\label{Lemma 20} Let $T^{n+1}$ be a manifold with
boundary $F^n$, lying in Euclidean space $E^{n+1}$, and with
$F^n$ connected. If the homomorphism $h: \pi_1(F^n) \to
\pi_1(T^{n+1})$, induced by the inclusion of $F^n$ in $T^{n+1}$,
is an isomorphism [1-1 and onto], then in the universal covering
$\tilde T^{n+1}$ of the polyhedron $T^{n+1}$ the polyhedron
$\tilde F^n$ covering $F^n$ is connected and simply connected.
If in addition $\pi_2(T^{n+1}) = 0$, then
$H^{fin}_2(\tilde T^{n+1}, J) = 0$.
\end{lem}
\begin{proof} Let $\kappa : \tilde T^{n+1} \to T^{n+1}$ be the
covering map; then $\kappa^{-1}(F^n) = \tilde F^n$ is, evidently,
the union of a finite or countable number of (connected)
manifolds. We show that $\tilde F^n$ is connected. [Just lift a
path between the images of two points. This reduces it to the
case of connecting two points $\tilde a_1, \tilde A_2 \in
\kappa^{-1}(A), A\in F^n$. Then there is a loop in $T^{n+1}$ at $A$
such that its lift to $\tilde A_1$ is a path to $\tilde A_2$
Since $\pi_1(F^n) \to \pi_1(T^{n+1})$ is {\em onto}, the loop in
$T^{n+1}$ is homotopic to a loop in $F^n$ which lifts to a path
in $\tilde F^n$ connecting $\tilde A_1, \tilde A_2$.]

Now let $\tilde \lambda$ be a closed path in $\tilde F^n$ based
at $\tilde A$, $\kappa(\tilde \lambda) = \lambda, \kappa(\tilde A) =
A$; then $\lambda$ is homotopic to the trivial loop in $T^{n+1}$.
Since $h$ is 1-1, $\lambda$ is nonhomotopic to the trivial loop in
$F^n$; but then $\tilde \lambda$ is homotopic to the trivial loop
in $\tilde F^n$.

Finally, $\pi_2(\tilde T^{n+1}) = \pi_2(T^{n+1}) = 0$,
$\pi_1(\tilde T^{n+1}) = 0$, and by Hurewicz's theorem (cf.,
for example \cite[p. 57]{H}), $H^{fin}_2(\tilde T^{n+1}, J) = 0$.
\end{proof}

\noindent 3.
\begin{thm} [= Theorem \ref{Theorem 1}] Let $F^n$ be a surface of class
$F_R$ in Euclidean space $E^{n+1}$ and suppose $H_1(F^n, J_2)
= 0$. Then $\kappa(F) \ge \kappa_1R$.
\end{thm}
\begin{proof} According to Poincar\'e duality $H_{n-1}(F^n,J_2)
= H_1(F^n,J_2) = 0$ (\cite[p. 484 $3.33_2$]{A1}). Applying Alexander
duality to the polyhedron $F^n \subset E^{n+1}$ (\cite[p. 490, 4:13]{A1}), we are led to 
$$H_1(E^{n+1} \setminus F^n, J_2) = H_{n-1}(F^n, J_2) = 0.$$
But by the Jordan-Brouwer theorem (\cite[p. 519, 3:44]{A1}),
$T^{n+1}$ is a connected component of $E^{n+1} \setminus
F^n$, from whence $H_1(T^{n+1}, J_2) = 0$.

[There is a more direct argument that $H_1(F^n, J_2) = 0
\Rightarrow H_1(T^{n+1}, J_2) = 0$. Suppose we have a 1-cycle
mod 2 in $T^{n+1}$; that is, a formal sum of loops $z^1$. We can
fill a loop in $E^{n+1}$ with a surface $S$ which can be assumed
to have general position relative to $F^n$. The intersection of
that surface with $F^n$ then consists of several loops which
form the boundary of the inside $S\cap T^{n+1}$ except for the
given loop. Each of those loops in $S\cap F^n$ is the boundary of
a surface in $F^n$ since $H_1(F^n, J_2) = 0$, and if we replace
the outside $S \setminus T^{n+1}$ by these surfaces in $F^n$
we get a surface in $T^{n+1}$ whose boundary is the original
loop.]

Thus, conditions 4), 5) of Lemma \ref{Lemma 19} are satisfied. Moreover,
$H_1(F^n, J_2) = 0 \Rightarrow H_1(F^n, J)$ has no $J$-summand.
(\cite[p.358, theorem 4:41]{A1}.) By Alexander duality then
$H_{n-1}(T^{n+1}, J)$ has no $J$-summand (\cite[p. 490, 4:1]{A1});
the torsion group $\Theta_{n-1}(T^{n+1})$ is always trivial
(cf. the corollary of 4:1 immediately after the formulation of 4:1,
\cite[p. 490]{A1}). Hence $H_{n-1}(T^{n+1}, J) =0$, and condition 6)
of Lemma \ref{Lemma 19} holds.

Conditions 1), 3) hold by an obvious means. Finally, condition 2)
follows from the theorem of Jordan-Brouwer.

By Lemma \ref{Lemma 19} $T^{n+1}$ cannot contain a normally imbedded
3-complex, so that either the central set of $F^n$ has points
of multiplicity $>3$ and hence $\kappa(F^n) \ge \kappa_1R$, or the
cutlocus has focal points and $\kappa(F^n) \ge R$. This completes
the proof of Theorem \ref{Theorem 1}.
\end{proof}
\begin{thm} [= Theorem \ref{Theorem 2}] Let $F \in F_R$. If the homomorphism
$h: \pi_1(F) \to \pi_1(T)$ induced by the inclusion of $F$ in
$T$ is an isomorphism, and $\pi_2(T) = 0$, then $\kappa(F) \ge
\kappa_1R$.
\end{thm}

{\sloppy
\begin{proof} Let the multiplicity of $Z^n$ equal 3. According
to Lemma \ref{Lemma 5}, in the universal covering $\tilde T^{n+1}$ of the
polyhedron $T^{n+1}$ there is contained a cutlocus $\tilde Z^n$,
normally imbedded in $\tilde T^{n+1}$as a 3-complex.

Therefore, as in the proof of Theorem \ref{Theorem 1}, it suffices to verify for
$\tilde T^{n+1}$that properties 1) -- 6) of Lemma \ref{Lemma 19} hold.

From Lemma \ref{Lemma 20} it follows that $\tilde F^n$ is connected and
simply-connected; therefore conditions 1), 2), 4) hold. Conditions
3) and 5) hold in view of the simple-connectedness of
$\tilde T^{n+1}$. It remains to verify condition 6). According to
Lemma \ref{Lemma 20}, $H^{fin}_2(\tilde T^{n+1}, J) =0$ We apply Poincar\'e
duality for infinite manifolds to $\tilde T^{n+1}$ (cf, e.g., \cite[\S\S 9, 33]{E}), accounting for condition 5) of Lemma \ref{Lemma 19}; we obtain
$H^{inf}_{n-1}(\tilde T^{n+1}, J) = 0$, that is, condition 6) also
holds, which concludes the proof of the theorem.
\end{proof}

}

In the case $n=2$ (of a surface $F^2$ in 3-dimensional space)
the condition of Theorem \ref{Theorem 2} can be significantly weakened.

\begin{thm} [= Theorem \ref{Theorem 3}] Let $n=2$, $F \in F_R$, and
$h: \pi_1(F) \to \pi_1(T)$ be onto. Then $\kappa(F) \ge \kappa_1R$.
\end{thm}

\medskip
\hrule\smallskip
\noindent p.~174.
\medskip

\begin{proof} We verify the conditions for applying Lemma \ref{Lemma 19}
to $\tilde T^3$. For $n=2$ condition 6) of Lemma \ref{Lemma 19} is found to
be unnecessary; concerning this, this condition was needed to
prove the nonexistence of manifolds $\tilde Z^{n-1}_j$ of the
third class (Lemma \ref{Lemma 11}). But a manifold $\tilde Z^1_j$ of the
third class would need to be a simply closed curve, since
otherwise [the holonomy would be trivial]. Consequently, it may
be assumed that $\tilde Z^1_j$ is a {\em finite} cycle of
$\tilde Z^2$. The orientability of $\tilde Z^1_j$ is evident, and
the exclusion of manifolds of the third follows from the
triviality of $H^{fin}_1(\tilde T^3, J)$ for a simply-connected
polyhedron $\tilde T^3$.

Moreover, conditions 2) and 4) are also found to be unnecessary.
Concerning this, in the proof of Lemma \ref{Lemma 19} conditions 2) and 4)
were used only in point a), in order to claim the orientability of
$\tilde Z^{n-1}_j$, which for $n=2$ holds automatically.

Conditions 3, 5) hold for simply-connected polyhedron $\tilde T$,
and it is only needed to verify condition 1). But for the proof of
connectedness of $\tilde F^n$ in Lemma \ref{Lemma 20} only the {\em ontoness}
of the homomorphism $h$ was used, which is assumed for
Theorem \ref{Theorem 3}.
\end{proof}

\section{Applications to some simple types of surfaces}\label{sec:applications}

\noindent 1. We consider now several particular cases, presenting interest
from a geometrical point of view. We give a definition of a
{\em solid homeomorphic to a ball with k handles}.

Let $K^{n+1}$ be a regular closed ball in $E^{n+1}$ and $f_j,
j=1,\ldots, k$ be homeomorphisms from $K^n \times [0,1]$
into $E^{n+1}$ such that the sets $Q_j = f_j(K^n\times (0,1))
\subset E^{n+1} \setminus K^{n+1}$, $\bar Q_j$ are pairwise
nonintersecting, and $[f_j(K^n \times 0) \cap f_j(K^n \times 1)]
\subset K^{n+1}, j=1,\ldots, k$. The polyhedron $h^{n+1}_k =
K^{n+1} \cap \bigcap^k_{j=1} Q_j$ is called an
$(n+1)$-dimensional ball with $k$ handles, and the boundary of
$h^{n+1}_k$ in $E^{n+1}$ is called an $n$-dimensional sphere
with $k$ handles and is designated by $S^n_k$.

A {\em regular (n+1)-dimensional toroidal ring $Tr^{n+1}$} is
the direct product of the disk $K^2$ by $E^{n-1}$ [sic. Should 
this be $S^{n-1}$? or $(S^1)^{n-1}$?]

Finally, we designate the $n$-dimensional sphere by $S^n$.

We now introduce the following classes of surfaces (cf. \S \ref{sec:introduction}): $\mathcal{S}$ consists of all surfaces of class $F_R$
homeomorphic to $S^n$; $\mathcal{H}_k$ consists of all surfaces of
class $F_R$ homeomorphic to $S^n_k$; $\mathcal{H}^0_k$ consists
of all surfaces of class $F_R$ bounding a solid homeomorphic 
to $h^{n+1}_k$; $Tr^0$ consists of all surfaces of class $F_R$
bounding a solid homeomorphic to $Tr^{n+1}$.

We recall that $\kappa(M) = \inf_{F \in M} \kappa(F), M\subset F_R$.

\noindent 2.
\begin{thm} [= Theorem \ref{Theorem 4}] $\kappa(\mathcal{S}) \ge \kappa_1R$;
$\kappa(\mathcal{H}^0_k) \ge \kappa_1R$; $\kappa(Tr^0) \ge \kappa_1R$.

For $n=2$ these inequalities reduce to equality.
\end{thm}

\begin{proof} In this part of the work we limit ourselves
to the proofs of the inequalities $\kappa(\mathcal{S}) \ge
\kappa_1R$, $\kappa(\mathcal{H}^0_k) \ge \kappa_1R$,
$\kappa(Tr^0) \ge \kappa_1R$; the sharp bound for classes
$\mathcal{S}$, $\mathcal{H}^0_k$ for $n=2$ and any $k = 1, 2, \ldots$
will be established in the second part of the work by the
construction of corresponding examples (we note that for
$n = 2$ we have $Tr^0 = \mathcal{H}^0_1$).

\medskip
\hrule\smallskip
\noindent p.~175
\medskip

\begin{description}
\item[a)] If $F^n \in \mathcal{S}$, then the assertion of the theorem
follows from Theorem \ref{Theorem 1}.
\item[b)] If $F^n \in \mathcal{H}^0_k \, (k>0)$, then, as is easily
seen, $F^n$ contains a subset $T^1$, homeomorphic to the
union of $k$ circles with one common point, such that there
exists a deformation 
$$\omega_t: T^{n+1} \times [0,1] \to T^{n+1},$$
$\omega_t(P) = P \, (P \in T^1, 0 \le t \le 1)$, of the identity
map $\omega_0$ into $\omega_1$, $\omega_1(T^{n+1})
\subset T^1$. For $n = 2$ thus it follows that $h$ is onto
and Theorem \ref{Theorem 3} can be applied.

Hence it follows that the homomorphism $h : \pi_1(F^n) \to
\pi_1(T^{n+1})$, induced by $F^n \subset T^{n+1}$, is an
isomorphism; $\pi_2(T^{n+1}) = \pi_2(T^1)$. Considering
the universal covering polyhedron $\tilde T^1$, it is easy
to convince oneself that $\pi_2(\tilde T^1) = 0$, whence
$\pi_2(T^1) = 0$; hence Theorem 2 can be applied to a surface
$F^n \in \mathcal{H}^0_k$, which leads to the required bound.

[$h$ is not an isomorphism for $n = 2$, only onto, but that
case has been covered.]

\item[c)] Let $F^n \in Tr^0$. We construct the universal
covering polyhedron $\tilde T^{n+1}$ for the solid $T^{n+1}$,
bonded by $F^n$ in $E^{n+1}$. Evidently, $\tilde T^{n+1}$ is
homeomorphic to $K^2 \times E^{n-1}$, $\tilde F^n$ is
homeomorphic to $S^1 \times E^{n-1}$ [the text was written
$E^{n+1}$]. We verify the conditions for the applicability
of Lemma \ref{Lemma 19} to $\tilde T^{n+1}$.
\end{description}

Evidently, all conditions besides 4) hold. But 4) was used
only in point a) of the proof of Lemma \ref{Lemma 19} for establishing
the orientability of manifolds $\tilde Z^{n-1}_j$. By Lemma
6, 5), for this it suffices to prove the orientability of the
manifolds $\tilde F^{n-1}_j \subset \tilde F^n$.

By a small isotopic deformation $\tilde F^{n-1}_j$ can be
moved to general position relative to the cycle $z^{n-1} =
P \times E^{n-1}$, $P\in S^1$ (the construction of such a
deformation is simplified thanks to the special from of
$z^{n-1}$). We also desigante the cycle obtained as a result
of the deformation by $\tilde F^{n-1}_j$ and we note that
in the deformation the characteristic of orientability of
$\tilde F^{n-1}_j$ is not changed. We construct, furthermore,
a simplicial subdivision of $\tilde F^n$, subcomplexes of
which are $\tilde F^{n-1}_j$ and $z^{n-1}$. since $z^{n-1}$
is the unique basis of homology cycles mod 2 of the
polyhedron $\tilde F^n$, there exists a chain $c^n$ (infinite)
mod 2 constructed in the above subdivision, such that
\begin{equation}\partial c^n = \tilde F^{n-1}_j + z^{n-1}. \label{eq:29}
\end{equation}

We take an integral chain $z^{n-1}_*$, which in mod 2 reduces
to $z^{n-1}$. For each component of string connectedness
$c^n_\alpha$ of the chain $c^n$ we choose an orientation of
the simplices of $C^n_\alpha$ such that for the integral chain $c^n_{\alpha *}$ obtained 
$$\partial c^n_\alpha = \tilde F^{n-1}_{j\alpha *} +
c^{n-1}_{\alpha *}, $$
where $\tilde F^{n-1}_{j\alpha *}$ consists of oriented
simplices of $\tilde F^{n-1}_j$, and $z^{n-1}_{\alpha *}$
consists of oriented simplices of $z^{n-1}$. It can be achieved
in this that the orientation of the simplices of
$z^{n-1}_{\alpha *}$ should be in accord with the orientation
fixed above $z^{n-1}$ of the cycle $z^{n-1}$; concerning this,
$c^n_\alpha$ lies in one of the two domains into which
$z^{n-1}$ separates $\tilde F^n$, and therefore for all oriented
simplices $\zeta^n$ of the chain $c^n_{\alpha *}$ and the
simplices $\zeta^{n-1}$ of the chain $z^{n-1}_*$ adjacent
with the coefficients of incidence $[\zeta^n : \zeta^{n-1}]$
are 1 throughout.

Since for each simplex of $\tilde F^{n-1}_j$ and $z^{n-1}$
the incidence equals 1 with simplices of $c^n$, the chains
$\tilde F^{n-1}_{j\alpha *}$, $z^{n-1}_{\alpha *}$,
$\tilde F^{n-1}_{\beta *}$, $z^{n-1}_{\beta *}$ for
$\alpha \ne \beta$ do not have simplices in common. We put
$c^n_* = \sum_\alpha c^n_{\alpha *}$; then
\begin{equation}\partial c^n_* = \sum_\alpha \tilde F^{n-1}_{j\alpha *} +
\sum_\alpha z^{n-1}_{\alpha *}. \label{eq:30}
\end{equation}

Inasmuch as $c^n_*$ contains all simplices of $c^n$, from
(\ref{eq:29}) and (\ref{eq:30}) it follows that
$$\tilde F^{n-1}_{j*} = \sum_\alpha \tilde F^{n-1}_{j\alpha *},
\quad z^{n-1}_{**} = \sum_\alpha z^{n-1}_{\alpha *}$$
in reduction mod 2 is transformed, correspondingly, to
$\tilde F^{n-1}_j$ and $z^{n-1}$ Due to the choice of
orientation of the chains $c^n_{\alpha *}$, $z^{n-1}_{**}$
coincides with $z^{n-1}$, and is, consequently, an integral
cycle. But then from (\ref{eq:30}) it follows that also
$\tilde F^{n-1}_{j*}$ is an integral cycle, which proves
the orientability of $\tilde F^{n-1}_j$.
\end{proof}

\noindent Submitted 30.VI.1961

\end{document}